\documentclass[12pt]{article}

\usepackage{amsmath}
\usepackage{amsthm}
\usepackage{amssymb}
\usepackage{natbib}

\theoremstyle{definition}
\newtheorem{theorem}{Theorem}
\newtheorem{corollary}[theorem]{Corollary}
\newtheorem{lemma}[theorem]{Lemma}

\theoremstyle{definition}
\newtheorem{definition}{Definition}
\newtheorem{example}{Example}

\newcommand{\mbR}{{\mathbb{R}}}
\newcommand{\mbZ}{{\mathbb{Z}}}

\newcommand{\Rd}{\partial}

\newcommand{\cU}{\mathcal{U}}

\begin{document}

\title{
Efron's curvature of the structural gradient model}
\author{Tomonari SEI}

\maketitle

\begin{abstract}
The structural gradient model is a multivariate statistical model
in order to extract various interactions of given data set.
In this note, we show that Efron's statistical curvature of
the structural gradient model is less than
that of a competitive mixture model under
a null hypothesis.
\end{abstract}

\setlength{\baselineskip}{18pt}

\section{Introduction} \label{section:intro}

Exponential families are important in statistical modeling.
For example, the Gaussian family and its subfamilies are often used
in multivariate analysis, time-series analysis, geostatistics
and any other areas that deal with quantitative data.
Using the exponential family is reasonable
because it is derived from the maximum entropy criterion
(see e.g.\ \cite{cover2006}).
It is also compatible with regression problem,
that is, the generalized linear models (\cite{GLM}).
A comprehensive book on exponential families is \cite{barndorff-nielsen1978}.

A drawback of exponential families is that
the probability density function is sometimes not explicitly expressed
due to the normalizing constant.
For example, if one would try to find three-dimensional interaction
of given data, a corresponding exponential family
is not available in explicit form.
Although the Markov Chain Monte Carlo procedure is available,
it requires some adjustment for convergence.
As an attempt to overcome the difficulty,
\cite{sei2010} suggested a new parametric family called 
a structural gradient model (SGM) for multivariate quantitative data.
SGM is numerically shown to have a desirable performance for such a purpose.
However, it is not known whether SGM is close to an exponential family or not.
In this paper, we give a partial answer to this problem.

A measure of closeness to an exponential family
is Efron's statistical curvature $\gamma^2$,
refered to {\em the Efron curvature} below.
It is defined in terms of the second-order derivative
of the log-likelihood function.
See Section~\ref{section:Efron} for the precise definition.
\cite{efron1975} showed that information loss of the maximum likelihood estimator
is asymptotically expanded as
$\gamma^2+{\rm o}(1)$ if the sample size $N$ goes to infinity.
It is known that $\gamma^2$ vanishes
if the model is an exponential family.
Furthermore, $\gamma^2$ is an intrinsic quantity,
that is, independent of the parameterization of the model.

Consider two statistical models $M_1$ and $M_2$,
and assume that they have a common density $p_0$
and a common score vector at $p_0$.
The Fisher information matrix at $p_0$ is common
in both models.
Then we can say that, without subjectivity, the model $M_1$
is closer to exponential family at $p_0$ than $M_2$
if the Efron curvature of $M_1$ is smaller than $M_2$.

We compare the Efron curvature of SGM and MixM,
which is a competitive model with SGM.
MixM is an abbreviation of {\em the structural mixture model}.
Here we briefly describe SGM and MixM.
For details, refer to Section~\ref{section:SGM} and \cite{sei2010}.
SGM is a statistical model on hypercube
represented by Fourier-expanded optimal transport between
the target density and the uniform density.
Here the Fourier coefficients are the unknown parameter.
The model is related to the optimal transport theory.
See \cite{villani2003} and \cite{villani2009} for the optimal transport theory.
MixM is represented by Fourier expansion of the probability density
function itself.
Both SGM and MixM do not need computation of normalizing constants,
in contrast to the exponential family.
We show that the curvature of SGM is less than MixM
under the common null hypothesis.
In other words, SGM is closer to exponential family than MixM.
This motivates to use SGM rather than MixM
for analyzing complicated dependency of given data.

The paper is organized as follows.
We recall the definition of the Efron curvature in Section~\ref{section:Efron}
and define SGM and MixM in Section~\ref{section:SGM}.
Then we state the main result of this paper
in Section~\ref{section:main}.
We give some discussion in Section~\ref{sec:discussion}.
Proofs are given in Section~\ref{section:proof}.

\section{Efron's statistical curvature} \label{section:Efron}

We recall the Efron curvature of a general statistical model
according to \cite{efron1975}, \cite{reeds1975} and \cite{amari1985}.
Intuitively, the Efron curvature is 
the residual when the second derivative of the log-likelihood
is projected onto the linear space spanned by
the score functions and the constant function.

Consider a parametric family
of density functions $p(x|\theta)$ with respect to
a base measure $\mathrm{d}x$ indexed by
a parameter vector $\theta=(\theta_u)_{u\in\cU}$,
where $\cU$ is a finite set.
Typically $\cU=\{1,\ldots,d\}$ with some $d\geq 1$,
but we will consider other case in the next section.
The parameter space $\Theta$ of $\theta$
is an open subset of $\mbR^{\cU}$,
where $\mbR^{\cU}$ denotes
the set of all real vectors $(\theta_u)_{u\in\cU}$
indexed by $\cU$.
Without loss of generality, we assume $0\in\Theta$ and define
the curvature at $\theta=0$.

Denote the first and second derivative of the log-likelihood function by
\begin{align*}
 L_u &= L_u(x) = \left.\frac{\Rd}{\Rd\theta_u}\log p(x|\theta)\right|_{\theta=0},
 \\
 L_{uv} &= L_{uv}(x) = \left.\frac{\Rd^2}{\Rd\theta_u\Rd\theta_v}\log p(x|\theta)\right|_{\theta=0}
\end{align*}
for $u,v\in\cU$.
Define the Fisher information $(J_{uv})_{u,v\in\cU}$
and the e-connection coefficients $(\Gamma_{uv,w})_{u,v,w\in\cU}$
and $(\Gamma_{uv}^{w})_{u,v,w\in\cU}$ by
\begin{align*}
 J_{uv}
 = \int p(x|0)
 L_uL_v\mathrm{d}x,
 \quad
 \Gamma_{uv,w}
 = \int p(x|0)
 L_{uv}L_w\mathrm{d}x,
 \quad
 \Gamma_{uv}^w
  = \sum_{s\in\cU}\Gamma_{uv,s}J^{sw},
\end{align*}
where $(J^{sw})$ is the inverse matrix of $(J_{sw})$.
We define a fourth-order tensor by
\[
 Q_{uv,wz}\ =\ 
 \int p(x|0)\left(L_{uv}+J_{uv}-\sum_{s\in\cU}\Gamma_{uv}^sL_s\right)
 \left(L_{wz}+J_{wz}-\sum_{t\in\cU}\Gamma_{wz}^tL_t\right)\mathrm{d}x.
\]
Finally, we define the Efron curvature by
\begin{align}
 \gamma^2\ =\ \sum_{u,v,w,z\in\cU}Q_{uv,wz}J^{uw}J^{vz}.
 \label{eqn:Efron-curvature}
\end{align}
The Efron curvature is a non-negative scalar quantity
independent of parameterization of $p(x|\theta)$.

The Efron curvature is related to the exponential family
and information loss as stated in Section~\ref{section:intro}.
Precise statements are as follows.
Recall that a statistical model $p(x|\theta)$
is called an exponential family (in canonical form) if
it is written as
$p(x|\theta)=\exp(\sum_{u\in\cU}\theta_ut_u(x)-\psi(\theta))$
with the sufficient statistics $t_u(x)$ and the normalizing function $\psi(\theta)$.

\begin{lemma}
 Let $\Theta$ be an open subset of $\mbR^{\cU}$.
 Then the Efron curvature vanishes over $\Theta$
 if and only if $p(x|\theta)$ is an exponential family.
\end{lemma}

\begin{lemma} \label{lem:loss}
 Let $(x_1,\ldots,x_N)$ be an i.i.d.\ sample
 from a density $p(x|\theta)$.
 Then, under some regularity conditions,
 the information loss of the maximum likelihood estimator $\hat{\theta}_N$
 is asymptotically
 \begin{align*}
  J_{uv}^{(x_1,\ldots,x_N)} - J_{uv}^{\rm \hat{\theta}_N}
 = \sum_{w,z}Q_{uw,vz}J^{wz} + {\rm o}(1)
 \end{align*}
 as $N\to\infty$,
 where $J_{uv}^{T}$
 denotes the Fisher information matrix of a statistic $T$.
 Note that $J_{uv}^{(x_1,\ldots,x_N)}=NJ_{uv}$.
 In particular, averaged information loss is given by
 \begin{align*}
  \sum_{u,v\in\cU}J^{uv}
  \left(J_{uv}^{(x_1,\ldots,x_N)} - J_{uv}^{\rm \hat{\theta}_N}
  \right)
 = \gamma^2 + {\rm o}(1).
 \end{align*}
\end{lemma}

For the proof, refer to \cite{efron1975}, \cite{reeds1975} and \cite{amari1985}.

\section{SGM and MixM} \label{section:SGM}

We prepare some notations to define SGM and MixM.
Let $m$ be a positive integer.
Denote the gradient operator and Hessian operator on $\mbR^m$
by $D=(\Rd/\Rd x_i)_{i=1}^m$ and
$D^2=(\Rd^2/\Rd x_i\Rd x_j)_{i,j=1}^m$, respectively.
The determinant and trace of a square matrix $A$
are denoted by $\det A$ and $\mathrm{\rm tr}A$, respectively.
For square matrices $A$ and $B$,
if $A-B$ is non-negative definite,
we write $A\succeq B$.
Let $\mbZ$ and $\mbZ_{\geq 0}$ be the set of all integers and 
all non-negative integers, respectively.
Let $(\mbZ_{\geq 0}^m)^+=\mbZ_{\geq 0}^m\setminus\{0\}$ be
the set of all $m$-dimensional non-negative integer vectors
except for zero vector.
Define $\|u\|=(\sum_{j=1}^m u_j^2)^{1/2}$ for $u\in\mbZ^m$.
The vectors are considered as column vectors
unless otherwise stated.

We give the definition of SGM and MixM.
Examples are given later.

\begin{definition}[SGM] \label{defn:SGM}
 Let $\cU$ be a finite subset of $(\mbZ_{\geq 0}^m)^+$.
 {\em The structural gradient model} (SGM) is
 a set of probability densities on the hypercube $[0,1]^m$
 with parameter vector $\theta=(\theta_u)\in\mbR^{\cU}$
 defined by
 \begin{align}
  p^{\rm (sgm)}(x|\theta)\ =\ \det(D^2\psi(x|\theta)),\quad 
  \psi(x|\theta)
   \ =\ \frac{1}{2}x^{\top}x
  - \sum_{u\in\cU} \frac{\theta_u}{\pi^2}\prod_{j=1}^m \cos(\pi u_jx_j).
  \label{eqn:SGM}
 \end{align}
 The parameter vector $\theta$ is said to be feasible if
 $D^2\psi(x|\theta)\succeq 0$ for every $x\in[0,1]^m$.
\end{definition}

\begin{definition}[MixM] \label{defn:mixture}
 Under the same notation as SGM, define
 \begin{align}
  p^{\rm (mix)}(x|\theta)\ =\ 
  1 + \sum_{u\in\cU} \theta_u\|u\|^2\prod_{j=1}^m \cos(\pi u_jx_j).
  \label{eqn:mix}
 \end{align}
 The set of $p^{\rm (mix)}(x|\theta)$ is called MixM in this paper.
 The parameter vector $\theta$ is feasible
 if $p^{\rm (mix)}(x|\theta)\geq 0$ for all $x\in[0,1]^m$.
\end{definition}

Remark that both $p^{\rm (sgm)}(x|\theta=0)$
and $p^{\rm (mix)}(x|\theta=0)$ are the uniform density.

Define a matrix $H_u(x)$ by
\begin{align}
  H_u(x)\ :=\ D^2\left(-\pi^{-2}\prod_{j=1}^m\cos(\pi u_jx_j)\right).
   \label{eqn:H_u}
\end{align}
In particular,
\begin{align*}
 \mathop{\rm tr}H_u(x)=\|u\|^2\prod_{j=1}^m\cos(\pi u_jx_j).
\end{align*}
Then we can rewrite (\ref{eqn:SGM}) and (\ref{eqn:mix}) as
\begin{align}
 p^{\rm (sgm)}(x|\theta)
 = \det\left(I+\sum_{u\in\cU}\theta_uH_u(x)\right),
 \quad p^{\rm (mix)}(x|\theta)
 = 1 + \sum_{u\in\cU}\theta_u\mathrm{tr}H_u(x).
 \label{eqn:SGM-MixM-rewrite}
\end{align}

We state a fundamental lemma.
For completeness, we prove it in Section~\ref{section:proof}.
We denote the indicator function of a set $A$
by $1_A$.

\begin{lemma}[\cite{sei2010} Lemma 3] \label{lem:score-fisher}
 The score vector at $\theta=0$ of both SGM and MixM
 is $(\mathop{\rm tr}H_u(x))_{u\in\cU}$.
 The common Fisher information matrix $J=(J_{uv})_{u,v\in\cU}$ at $\theta=0$
 is $J_{uv}=\|u\|^42^{-|\sigma(u)|}1_{\{u=v\}}$,
 where $\sigma(u)=\{j\in\{1,\ldots,m\}\mid u_j>0\}$ and $|\sigma(u)|$ denotes the cardinality of $\sigma(u)$.
 In particular, $J_{uv}$ is diagonal.
\end{lemma}

We give a few examples,
where we write $(u_1,\ldots,u_m)$ instead of $(u_1,\ldots,u_m)^{\top}$ for simplicity.

\begin{example}
 Let $m=2$ and $\cU=\{(1,1)\}$.
 We abbreviate $\theta_{(1,1)}$ as $\theta$ for simplicity.
 Then we have
 \begin{align*}
  p^{\rm (sgm)}(x|\theta)
  &= \det
 \begin{pmatrix}
  1+\theta\cos(\pi x_1)\cos(\pi x_2) & -\theta\sin(\pi x_1)\sin(\pi x_2) \\
  -\theta\sin(\pi x_1)\sin(\pi x_2) & 1+\theta\cos(\pi x_1)\cos(\pi x_2)
 \end{pmatrix}
  \\
  &= 1+2\theta \cos(\pi x_1)\cos(\pi x_2)
  +\theta^2\{\cos^2(\pi x_1)+\cos^2(\pi x_2)-1\}
 \end{align*}
 and $p^{\rm (mix)}(x|\theta)=1+2\theta \cos(\pi x_1)\cos(\pi x_2)$.
 SGM is feasible if and only if $|\theta|\leq 1$.
 MixM is feasible if and only if $|\theta|\leq 1/2$.
\end{example}

\begin{example}
Let $m=3$ and $\cU=\{(1,0,0),(2,0,0),(1,1,0),(2,1,0),(1,1,1)\}$.
Then the diagonal part $J_u:=J_{uu}$ of the Fisher information matrix
is
\[
 J_{(1,0,0)}=\frac{1}{2},\quad
 J_{(2,0,0)}=8,\quad
 J_{(1,1,0)}=1,\quad
 J_{(2,1,0)}=\frac{25}{4},\quad
 J_{(1,1,1)}=\frac{9}{8}.
\]
\end{example}

\section{Main result} \label{section:main}

Consider a finite subset $\cU$ of $(\mbZ_{\geq 0}^m)^+$.
Let $(\gamma_{\cU}^2)^{\rm (sgm)}$ and $(\gamma_{\cU}^2)^{\rm (mix)}$
be the Efron curvature (\ref{eqn:Efron-curvature}) of SGM and MixM at $\theta=0$, respectively.
For each $i\in\{1,\ldots,m\}$,
we set $\mbZ_i=\{u\in(\mbZ_{\geq 0}^m)^+\mid u_j=0\ \mbox{if}\ j\neq i\}$.

Our main result is the following theorem.

\begin{theorem} \label{thm:main}
 For any finite $\cU\subset(\mbZ_{\geq 0}^m)^+$, the following inequality holds:
 \begin{align}
  0 < (\gamma_{\cU}^2)^{\rm (sgm)} \leq (\gamma_{\cU}^2)^{\rm (mix)}.
  \label{eqn:main}
 \end{align}
 Equality holds if and only if there is some $i\in\{1,\ldots,m\}$
 such that $\cU\subset\mbZ_i$.
 If the equality holds, then the two models coincide.
\end{theorem}

We give more explicit expression of the two quantities.
We prepare some additional notations.
For a vector $U=(U_i)\in\mbZ^m$,
its component-wise absolute value is denoted by
$\mathrm{abs}(U)=(|U_i|)$.
For two vectors $U=(U_i)$ and $V=(V_i)$,
their component-wise product (Hadamard product) is denoted by
$U\circ V=(U_iV_i)$.
Let $\beta=(\beta_i)\in\{-1,1\}^m$
be a Bernoulli sequence, that is,
$\beta_i$ independently takes the value $\pm 1$ with probability $1/2$ each.
For a Bernoulli sequence $\beta$ and a vector $u\in\cU$
we call the vector $U=\beta\circ u$
{\em Bernoulli randomization} of $u$.
The expectation with respect to $U$ (inherited from $\beta$)
is denoted as $\mathrm{E}_U$.
If Bernoulli randomization of two or more (possibly the same) vectors are considered,
then they are assumed to be independent.
Recall that $\|u\|=(\sum_{j=1}^mu_j^2)^{1/2}$
and $\sigma(u)=\{j\mid u_j>0\}$.

The explicit expression of the Efron curvature is given in the following theorem.
The inequality (\ref{eqn:main}) is obtained as a corollary.

\begin{theorem} \label{thm:main-explicit}
 The Efron curvature of SGM and MixM at $\theta=0$ is given by
 \begin{align}
  (\gamma_{\cU}^2)^{\rm (sgm)}
   &= \sum_{u,v\in\cU}\mathrm{E}_{U,V,\tilde{U},\tilde{V}}
   \left[\omega_{\cU}(U,V,\tilde{U},\tilde{V})
    2^{|\sigma(u)|+|\sigma(v)|}
    \frac{(U^{\top}V)^2(\tilde{U}^{\top}\tilde{V})^2}{\|u\|^4\|v\|^4}
   \right],
   \label{eqn:gamma-grad}
   \\
  (\gamma_{\cU}^2)^{\rm (mix)}
   &= \sum_{u,v\in\cU}\mathrm{E}_{U,V,\tilde{U},\tilde{V}}
   \left[\omega_{\cU}(U,V,\tilde{U},\tilde{V})2^{|\sigma(u)|+|\sigma(v)|}
   \right],
   \label{eqn:gamma-mix}
 \end{align}
 where $U,V,\tilde{U},\tilde{V}$ are 
 Bernoulli randomization of $u,v,u,v$, respectively, and
 \[
  \omega_{\cU}(U,V,\tilde{U},\tilde{V})
   = 1_{\{U+V+\tilde{U}+\tilde{V}=0,\ \mathrm{abs}(U+V)\notin\cU\cup\{0\}\}}.
 \]
 In particular, $(\gamma_{\cU}^2)^{\rm (sgm)}$ and $(\gamma_{\cU}^2)^{\rm (mix)}$ 
 are rational numbers.
\end{theorem}

Table~\ref{table:numeric} shows the Efron curvature
for several specific cases of $\cU$.
Let $e_i=(1_{\{j\neq i\}})_{j=1}^m$, the $i$-th unit vector.

\begin{table}[htbp]
 \caption{The Efron curvature for several cases of $\cU$.}
 \label{table:numeric}
 \begin{center}
  \begin{tabular}{c|c|c}
   $\cU$& $(\gamma_{\cU}^2)^{\rm (sgm)}$& $(\gamma_{\cU}^2)^{\rm (mix)}$ \\
   \hline
   $\{fe_i\}_{1\leq f\leq d,1\leq i\leq m}$& $2^{-2}d(d+1)m$& $2^{-2}d(d+1)m+d^2m(m-1)$ \\
   $\{e_i+e_j\}_{1\leq i<j\leq m}$& $2^{-5}m(m-1)(m+2)$& $2^{-3}m(m-1)(2m^2-6m+9)$\\
   $\{e_i+e_{i+1}\}_{i=1}^{m-1}$& $2^{-4}(7m-10)$& $2^{-2}(4m^2-3m-5)$\\
   $\{e_1+e_i\}_{i=2}^m$& $2^{-5}(m-1)(3m+2)$& $2^{-2}(m-1)(6m-7)$\\
  \end{tabular}
 \end{center}
\end{table}

We end with an asymptotic property.
For the first three examples in Table~\ref{table:numeric},
it is easily confirmed that $(\gamma_{\cU}^2)^{\rm (sgm)}/(\gamma_{\cU}^2)^{\rm (mix)}$
converges to $0$ as $m\to\infty$.
This property holds in a more general setting.
We define two sets $M(\cU)$ and $N(\cU)$ by
\begin{align*}
 M(\cU)
  &= \left\{
  (u,v)\in\cU^2\mid u+v\not\in\cU
  \right\},
 \\
 N(\cU)
 &= \left\{
  (u,v)\in\cU^2\mid \sigma(u)\cap\sigma(v)\neq\emptyset
  \right\}.
\end{align*}
We denote cardinality of a set $A$ by $|A|$.

\begin{theorem} \label{thm:asymptotics}
 Let $\cU_m$ be a finite subset of $(\mbZ_{\geq 0}^m)^+$ for each $m\in\{1,2,\ldots\}$.
 Assume that $\max_{u\in\cU_m}|\sigma(u)|$ is bounded over $m$.
 Further assume $|N(\cU_m)|/|M(\cU_m)|\to 0$ as $m\to\infty$.
 Then $(\gamma_{\cU_m}^2)^{\rm (sgm)}/(\gamma_{\cU_m}^2)^{\rm (mix)}\to 
 0$ as $m\to\infty$.
\end{theorem}

Let $\mu(\cU)$ be the set of maximal elements of $\cU$,
that is,
\[
 \mu(\cU) = \left\{
 u\in\cU\mid \forall v\in\cU\setminus\{u\},
 \ \exists i\in\{1,\ldots,m\}\ \mbox{s.t.}\ v_i<u_i
 \right\}.
\]

\begin{corollary} \label{coro:asymptotics-easy}
 Let $\cU_m$ be a finite subset of $(\mbZ_{\geq 0}^m)^+$ for each $m\in\{1,2,\ldots\}$.
 Assume that $\max_{u\in\cU_m}|\sigma(u)|$ is bounded over $m$.
 Further assume $|N(\cU_m)|/|\mu(\cU_m)|^2\to 0$ as $m\to\infty$.
 Then $(\gamma_{\cU_m}^2)^{\rm (sgm)}/(\gamma_{\cU_m}^2)^{\rm (mix)}\to 
 0$ as $m\to\infty$.
\end{corollary}

Table~\ref{table:N-and-mu} shows the numbers $|N(\cU)|$ and $|\mu(\cU)|$
for the examples in Table~\ref{table:numeric}.
It is consistent with Corollary~\ref{coro:asymptotics-easy},
that is, $|N(\cU)|/|\mu(\cU)|^2\to 0$ only for the first three cases.

\begin{table}[htbp]
 \caption{The numbers $|N(\cU)|$ and $|\mu(\cU)|$.}
 \label{table:N-and-mu}
 \begin{center}
  \begin{tabular}{c|c|c}
   $\cU$& $|N(\cU)|$ &$|\mu(\cU)|$ \\
   \hline
   $\{fe_i\}_{1\leq f\leq d,1\leq i\leq m}$& $d^2m$& $m$\\
   $\{e_i+e_j\}_{1\leq i<j\leq m}$& $2^{-1}m(m-1)(2m-3)$& $2^{-1}m(m-1)$\\
   $\{e_i+e_{i+1}\}_{i=1}^{m-1}$& $3m-5$& $m-1$\\
   $\{e_1+e_i\}_{i=2}^m$& $(m-1)^2$& $m-1$\\
  \end{tabular}
 \end{center}
\end{table}

\section{Discussion} \label{sec:discussion}

We evaluated the Efron curvature of SGM and MixM (Theorem~\ref{thm:main-explicit})
and used it to show that SGM has smaller curvature than MixM (Theorem~\ref{thm:main}).
Here we give some unsolved problems.

In Table~\ref{table:numeric}, we listed explicit formulas of the Efron curvature
for specific $\cU$'s by using (\ref{eqn:gamma-grad}) and (\ref{eqn:gamma-mix}).
It is challenging to derive formulas for more practical sets, such as
\[
 \cU = \left\{u\in(\mbZ_{\geq 0}^m)^+
 \mid \|u\|_1\leq 3,\ \|u\|_{\infty}\leq 2
 \right\},
 \quad \|u\|_1=\sum_{j=1}^mu_j,
 \quad \|u\|_{\infty}=\max_{j}u_j.
\]
\cite{sei2010} used this set to analyze multivariate datasets.
For each small $m$, we can evaluate the curvature by direct computation.
However, the computation needs exponential complexity
with respect to the dimension $m$
as long as one uses (\ref{eqn:gamma-grad}) and (\ref{eqn:gamma-mix}).
Combinatorial methods may solve the problem.

We studied the {\em averaged} curvature $\gamma^2$.
Instead, one can consider a tensor $H_{uv}:=\sum_{w,z}Q_{uw,vz}J^{wz}$ appearing in Lemma~\ref{lem:loss},
which is called the embedding e-curvature (\cite{amari1985}).
Although an inequality $H_{uv}^{\rm (sgm)}\preceq H_{uv}^{\rm (mix)}$ is conjectured
by numerical study, it could not be proved.

In this paper, we only considered the curvature at the origin $\theta=0$.
The reason that we restrict comes from two different kinds of difficulty.
One is conceptual difficulty:
the probability densities (and score vectors) of SGM and MixM are different
except at $\theta=0$.
An approach may be to consider a local mixture model of SGM
at each point $\theta$ (\cite{marriott2002}).
The another kind of difficulty is computational one.
The expression of the Efron curvature at $\theta\neq 0$ of SGM
seems complicated. Even the Fisher information matrix $J_{uv}$
is not written in elementary functions in general.
However, the expression is written
at least in terms of integration of multi-dimensional rational functions
because $p(x|\theta)$ is a polynomial of $\theta_u$ and $z_j=\mathrm{e}^{\mathrm{i}\pi x_j}$.
Algebraic methods on integration may be helpful.

\section{Proofs} \label{section:proof}

\subsection{Proof of Lemma~\ref{lem:score-fisher} and Theorem~\ref{thm:main-explicit}}

We calculate the Efron curvature of SGM and MixM step-by-step.

For SGM,
we denote the quantities $L_{uv}(x)$, $\Gamma_{uv,w}$, $\Gamma_{uv}^{w}$,
$Q_{uv,wz}$, $\gamma^2$ in Section~\ref{section:Efron} by
$L_{uv}^{\rm (sgm)}(x)$, $\Gamma_{uv,w}^{\rm (sgm)}$, 
$(\Gamma_{uv}^{w})^{\rm (sgm)}$, 
$Q_{uv,wz}^{\rm (sgm)}$, $(\gamma^2)^{\rm (sgm)}$, respectively.
Similarly, for MixM, we denote
$L_{uv}^{\rm (mix)}(x)$, $\Gamma_{uv,w}^{\rm (mix)}$, 
$(\Gamma_{uv}^{w})^{\rm (mix)}$, $Q_{uv,wz}^{\rm (mix)}$, 
$(\gamma^2)^{\rm (mix)}$.
We use $L_u(x)$ and $J_{uv}$ without superscripts
because they are common in both models.
Recall that a random matrix $H_u=H_u(x)$ is defined by (\ref{eqn:H_u}).

\begin{lemma} \label{lem:score-H}
 For any $u,v\in\cU$, the following equality holds:
 \[
  L_u(x) = \mathop{\rm tr}H_u,
 \quad L_{uv}^{\rm (sgm)}(x) = -\mathop{\rm tr}(H_uH_v)
 \quad L_{uv}^{\rm (mix)}(x) = -(\mathop{\rm tr}H_u)(\mathop{\rm tr}H_v).
 \]
\end{lemma}

\begin{proof}
 By (\ref{eqn:SGM-MixM-rewrite}),
 the log-likelihood of SGM and MixM
 are expanded around $\theta=0$ as
 \begin{align*}
 \log p^{\rm (sgm)}(x|\theta)
  &= \sum_{u\in\cU}\theta_u\mathop{\rm tr}H_u
  -\frac{1}{2}\sum_{u,v\in\cU}\theta_u\theta_v\mathop{\rm tr}(H_uH_v)
  +{\rm O}(\|\theta\|^3),
  \\
  \log p^{\rm (mix)}(x|\theta)
  &= \sum_{u\in\cU}\theta_u\mathop{\rm tr}H_u
  -\frac{1}{2}\sum_{u,v\in\cU}\theta_u\theta_v(\mathop{\rm tr}H_u)
  (\mathop{\rm tr}H_v)
  +{\rm O}(\|\theta\|^3).
 \end{align*}
 Then the result follows. 
\end{proof}

Since the random variables $L_u(x)$, $L_{uv}^{\rm (sgm)}(x)$ and 
$L_{uv}^{\rm (mix)}(x)$ are written in terms of $H_u$,
it is valuable to consider moment formulas of $H_u$.

\begin{lemma} \label{lem:Bernoulli}
 Let $u\in\cU$.
 Let $U$ be Bernoulli randomization of $u$.
 Then $H_u$ is written as $H_u=\mathrm{E}_U[\mathrm{e}^{\mathrm{i}\pi U^{\top}x}UU^{\top}]$.
 Furthermore, the random variable $x$ can be replaced with
 a random variable $\xi$ uniformly distributed on $[-1,1]^m$,
 when any moment of $\mathop{\rm tr}(H_u)$ and $\mathop{\rm tr}(H_uH_v)$
 is evaluated.
\end{lemma}

\begin{proof}
 By Euler's formula $\cos\phi=(\mathrm{e}^{\mathrm{i}\pi\phi}+\mathrm{e}^{-\mathrm{i}\pi\phi})/2$,
 we obtain
 \[
  \prod_{j=1}^m\cos(\pi u_jx_j) = \mathrm{E}_U[\mathrm{e}^{\mathrm{i}\pi U^{\top}x}].
 \]
 Therefore $H_u=\mathrm{E}_U[\mathrm{e}^{\mathrm{i}\pi U^{\top}x}UU^{\top}]$.
 Next we consider moments.
 Consider, for example, expectation of
 $\mathop{\rm tr}(H_uH_v)$.
 The other moments are similarly evaluated.
 Let $\tilde{\beta}$ be a Bernoulli sequence,
 which is independent of $x$ and any other Bernoulli sequences.
 Put $\xi=\tilde{\beta}\circ x$.
 Then $\xi$ has the uniform distribution on $[-1,1]^m$, and
 \begin{align*}
  \mathrm{E}_{\xi}[\mathop{\rm tr}(H_u(\xi)H_v(\xi))]
  &= \mathrm{E}_{\xi,U,V}[\mathrm{e}^{\mathrm{i}\pi U^{\top}\xi}
  \mathrm{e}^{\mathrm{i}\pi V^{\top}\xi}
  \mathop{\rm tr}(UU^{\top}VV^{\top})]
  \\
  &= \mathrm{E}_{\xi,U,V}[\mathrm{e}^{\mathrm{i}\pi U^{\top}\xi}
  \mathrm{e}^{\mathrm{i}\pi V^{\top}\xi}(U^{\top}V)^2]
  \\
  &= \mathrm{E}_{x,\tilde{\beta},U,V}
  [\mathrm{e}^{\mathrm{i}\pi (U\circ\tilde{\beta})^{\top}x}
  \mathrm{e}^{\mathrm{i}\pi (V\circ\tilde{\beta})^{\top}x}
  (U^{\top}V)^2]
  \\
  &= \mathrm{E}_{x,\tilde{U},\tilde{V}}
  [\mathrm{e}^{\mathrm{i}\pi \tilde{U}^{\top}x}
  \mathrm{e}^{\mathrm{i}\pi \tilde{V}^{\top}x}
  (\tilde{U}^{\top}\tilde{V})^2]
  \\
  &= \mathrm{E}_{x}[\mathop{\rm tr}(H_u(x)H_v(x))],
 \end{align*}
 where we put $\tilde{U}=U\circ\tilde{\beta}$ and $\tilde{V}=V\circ\tilde{\beta}$,
 and used an identity $\tilde{U}^{\top}\tilde{V}=U^{\top}V$.
\end{proof}

From Lemma~\ref{lem:Bernoulli}, we simply write
$H_u=\mathrm{E}_U[\mathrm{e}^{\mathrm{i}\pi U^{\top}\xi}UU^{\top}]$ below
and the expectation with respect to $x$
is replaced with the expectation with respect to $\xi$.
Note that $\mathrm{E}_{\xi}[\mathrm{e}^{\mathrm{i}\pi a^{\top}\xi}]=1_{\{a=0\}}$
for any $a\in\mbZ^m$.

Now the Fisher information matrix is evaluated as
\begin{align*}
 J_{uv}
 & 
 = \mathrm{E}_{\xi}[\mathop{\rm tr}H_u\mathop{\rm tr}H_v]
 \\
 &= \mathrm{E}_{\xi,U,V}[\mathrm{e}^{\mathrm{i}\pi (U+V)^{\top}\xi}\|u\|^2\|v\|^2]
 \\
 &= \mathrm{E}_{U,V}[1_{\{U+V=0\}}\|u\|^2\|v\|^2]
 \\
 &= 
 \mathrm{E}_{\beta,\tilde{\beta}}[
 1_{\{\beta\circ u=-\tilde{\beta}\circ v\}}]\|u\|^2\|v\|^2
 \\
 &= \mathrm{E}_{\beta,\tilde{\beta}}\left[\prod_{i=1}^m
   \{1_{\{u_i=v_i=0\}}+1_{\{u_i=v_i>0,\beta_i=-\tilde{\beta}_i\}}\}\right]\|u\|^2\|v\|^2
 \\
 &=  1_{\{u=v\}}2^{-|\sigma(u)|}\|u\|^4,
\end{align*}
where $\beta$ and $\tilde{\beta}$ are Bernoulli sequences.
This proves Lemma~\ref{lem:score-fisher}.
By similar computation, we have the following lemma.

\begin{lemma} \label{lem:e-connection}
 Let $U,V,S$ be Bernoulli randomization of $u,v,s\in\cU$.
 Then
 \begin{align*}
  (\Gamma_{uv}^w)^{\rm (sgm)}
  &= -\mathrm{E}_{U,V}[1_{\{\mathrm{abs}(U+V)=w\}}(U^{\top}V)^2\|w\|^{-2}],
  \\
  (\Gamma_{uv}^w)^{\rm (mix)}
  &= -\mathrm{E}_{U,V}[1_{\{\mathrm{abs}(U+V)=w\}}\|u\|^2\|v\|^2\|w\|^{-2}],
 \end{align*}
\end{lemma}

\begin{proof}
 We first calculate $\Gamma_{uv,s}^{\rm (sgm)}$.
 By Lemma~\ref{lem:score-H} and Lemma~\ref{lem:Bernoulli}, we have
 \begin{align*}
  \Gamma_{uv,s}^{\rm (sgm)}
  &= -\mathrm{E}_{\xi}[\mathop{\rm tr}(H_uH_v)\mathop{\rm tr}H_s]
   \\
  &= -\mathrm{E}_{\xi,U,V,S}
   \left[
    \mathrm{e}^{\mathrm{i}\pi (U+V+S)^{\top}\xi}
    (U^{\top}V)^2\|s\|^2
   \right]
   \\
  &= -\mathrm{E}_{U,V,S}
   \left[
    1_{\{U+V+S=0\}}(U^{\top}V)^2\|s\|^2
   \right].
 \end{align*}
 By using the expression of $\Gamma_{uv,s}^{\rm (sgm)}$ and $J^{sw}$, we have
 \begin{align*}
  (\Gamma_{uv}^w)^{\rm (sgm)}
   &= \sum_{s\in\cU}\Gamma_{uv,s}^{\rm (sgm)}J^{sw}
   \\
   &= -\sum_{s\in\cU}
   \mathrm{E}_{U,V,S}
   \left[
    1_{\{U+V+S=0\}}(U^{\top}V)^2\|s\|^2
   \right]
   1_{\{s=w\}}2^{|\sigma(s)|}\|s\|^{-4}
   \\
   &= 
  -\mathrm{E}_{U,V,W}
   \left[
    1_{\{U+V+W=0\}}(U^{\top}V)^2\|w\|^{-2}2^{|\sigma(w)|}
   \right]
   \\
  &= 
  -\mathrm{E}_{U,V,\beta}
   \left[
    1_{\{\mathrm{abs}(U+V)=w\}}1_{\{U+V=\beta\circ w\}}(U^{\top}V)^2\|w\|^{-2}
  2^{|\sigma(w)|}
   \right],
  \\
  &=
  -\mathrm{E}_{U,V}
   \left[
    1_{\{\mathrm{abs}(U+V)=w\}}(U^{\top}V)^2\|w\|^{-2}
   \right],
 \end{align*}
 where $\beta$ is a Bernoulli sequence.
 The expression of $\Gamma_{uv,s}^{\rm (mix)}$ and 
 $(\Gamma_{uv}^{w})^{\rm (mix)}$ is obtained similarly.
\end{proof}

\begin{lemma} \label{lem:Q-expression}
 The curvature tensor of SGM and MixM at $\theta=0$ is
 \begin{align*}
  Q_{uv,wz}^{\rm (sgm)}
   &= \mathrm{E}_{U,V,W,Z}
   \left[
    \omega_{\cU}(U,V,W,Z)
    (U^{\top}V)^2(W^{\top}Z)^2
   \right],
  \\
  Q_{uv,wz}^{\rm (mix)}
   &= \mathrm{E}_{U,V,W,Z}
   \left[
    \omega_{\cU}(U,V,W,Z)
    \|u\|^2\|v\|^2\|w\|^2\|z\|^2
   \right],
 \end{align*}
 respectively, where $U,V,W,Z$ are Bernoulli randomization of $u,v,w,z$
 and
 \[
 \omega_{\cU}(U,V,W,Z) = 1_{\{U+V+W+Z=0,\ \mathrm{abs}(U+V)\notin\cU\cup\{0\}\}}.
 \]
\end{lemma}

\begin{proof}
 We only derive the expression of $Q_{uv,wz}^{\rm (sgm)}$.
 The expression of $Q_{uv,wz}^{\rm (mix)}$ is obtained similarly.
 We first prove
 \begin{align}
  R_{uv}^{\rm (sgm)}(x)
  &:=L_{uv}^{\rm (sgm)}(x)
   +J_{uv}-\sum_{s\in\cU}(\Gamma_{uv}^s)^{\rm (sgm)} L_s(x)
  \nonumber\\
  &= -\mathrm{E}_{U,V}
  \left[
   1_{\{\mathrm{abs}(U+V)\notin\cU\cup\{0\}\}}\mathrm{e}^{\mathrm{i}\pi(U+V)^{\top}\xi}(U^{\top}V)^2
  \right].
  \label{eqn:R-expression}
 \end{align}
 The last term of $R_{uv}^{\rm (sgm)}(x)$ is
 \begin{align*}
  -\sum_{s\in\cU}(\Gamma_{uv}^s)^{\rm (sgm)}L_s
   &=
  \sum_{s\in\cU}\mathrm{E}_{U,V,S}
  \left[
   1_{\{\mathrm{abs}(U+V)=s\}}\mathrm{e}^{\mathrm{i}\pi S^{\top}\xi}
   (U^{\top}V)^2
  \right]
   \\
  &=
  \sum_{s\in\cU}\mathrm{E}_{U,V,\beta}
  \left[
  1_{\{\mathrm{abs}(U+V)=s\}}\mathrm{e}^{\mathrm{i}\pi (\beta\circ(U+V))^{\top}\xi}
  (U^{\top}V)^2
  \right]
  \\
   &=
  \mathrm{E}_{U,V,\beta}
  \left[
   1_{\{\mathrm{abs}(U+V)\in\cU\}}\mathrm{e}^{\mathrm{i}\pi (\beta\circ(U+V))^{\top}\xi}
   (U^{\top}V)^2
  \right]
   \\
   &=
  \mathrm{E}_{U,V}
  \left[
   1_{\{\mathrm{abs}(U+V)\in\cU\}}\mathrm{e}^{\mathrm{i}\pi (U+V)^{\top}\xi}
   (U^{\top}V)^2
  \right],
 \end{align*}
 where $\beta$ is a Bernoulli sequence.
 For the first and second term of $R_{uv}^{\rm (sgm)}(x)$, we have
 \begin{align*}
  L_{uv}^{\rm (sgm)}
  &= -\mathrm{E}_{U,V}
  \left[
  \mathrm{e}^{\mathrm{i}\pi(U+V)^{\top}\xi}(U^{\top}V)^2
  \right],
  \\
  J_{uv}
  &= \mathrm{E}_{U,V}
  \left[
  1_{\{U+V=0\}}(U^{\top}V)^2
  \right]
  = \mathrm{E}_{U,V}
  \left[
  1_{\{\mathrm{abs}(U+V)=0\}}\mathrm{e}^{\mathrm{i}\pi(U+V)^{\top}\xi}(U^{\top}V)^2
  \right].
 \end{align*}
 Hence (\ref{eqn:R-expression}) is obtained.
 Now the tensor $Q_{uv,wz}^{\rm (sgm)}$ is calculated as follows:
 \begin{align*}
  Q_{uv,wz}^{\rm (sgm)}
   &= \mathrm{E}_{\xi}[R_{uv}^{\rm (sgm)}R_{wz}^{\rm (sgm)}]
  \\
  &= \mathrm{E}_{\xi,U,V,W,Z}
  \left[
  1_{\{\mathrm{abs}(U+V)\notin\cU\cup\{0\},\mathrm{abs}(W+Z)\notin\cU\cup\{0\}\}}
  \mathrm{e}^{\mathrm{i}\pi(U+V+W+Z)^{\top}\xi}
  (U^{\top}V)^2(W^{\top}Z)^2
  \right]
  \\
  &= \mathrm{E}_{U,V,W,Z}
  \left[
  1_{\{U+V+W+Z=0,\mathrm{abs}(U+V)\notin\cU\cup\{0\}\}}
  (U^{\top}V)^2(W^{\top}Z)^2
  \right]
 \end{align*}
 Therefore we obtain the desired expression.
\end{proof}

 We finally prove Theorem~\ref{thm:main-explicit}.
 Since the Fisher information matrix is diagonal,
 we have
 \begin{align*}
  (\gamma^2)^{\rm (sgm)}
   &= \sum_{u,v,w,z\in\cU}Q_{uv,wz}^{\rm (sgm)}J^{uw}J^{vz}
   \\
   &= \sum_{u,v\in\cU}Q_{uv,uv}^{\rm (sgm)}J^{uu}J^{vv}
    \\
  &= \sum_{u,v\in\cU}
 \mathrm{E}_{U,V,\tilde{U},\tilde{V}}
 \left[
  \omega_{\cU}(U,V,\tilde{U},\tilde{V})
  (U^{\top}V)^2(\tilde{U}^{\top}\tilde{V})^2
 \right]
 \frac{2^{|\sigma(u)|+|\sigma(v)|}}{\|u\|^4\|v\|^4}.
 \end{align*}
 Thus (\ref{eqn:gamma-grad}) is proved.
 (\ref{eqn:gamma-mix}) is shown similarly.

\subsection{Proof of Theorem~\ref{thm:main}}

We prove Theorem~\ref{thm:main} by using
the explicit expression (\ref{eqn:gamma-grad}) and (\ref{eqn:gamma-mix})
of the Efron curvature.
We abbreviate $\gamma_{\cU}^2$ as $\gamma^2$.

 We prove the first inequality in (\ref{eqn:main}).
 By the expression (\ref{eqn:gamma-grad}),
 it is sufficient to show that $\omega_{\cU}(u,u,-u,-u)=1$
 for some $u\in\cU$.
 Let $u$ be an element such that $\|u\|_1=\max_{v\in\cU}\|v\|_1$.
 Then we have $u+u-u-u=0$ and $u+u\notin\cU\cup\{0\}$,
 and hence $\omega_{\cU}(u,u,-u,-u)=1$.

 The second inequality in (\ref{eqn:main}) follows from
 equations (\ref{eqn:gamma-grad}), (\ref{eqn:gamma-mix}),
 and
 \[
 (U^{\top}V)^2(\tilde{U}^{\top}\tilde{V})^2
 \ \leq\ \|U\|^2\|V\|^2\|\tilde{U}\|^2\|\tilde{V}\|^2
 \ =\ \|u\|^4\|v\|^4.
 \]
 We now consider the equality condition.
 First assume $\cU\subset\mbZ_i$. Then
 $(U^{\top}V)^2(\tilde{U}^{\top}\tilde{V})^2$ in (\ref{eqn:gamma-grad})
 is equal to $(u_iv_i)^2(u_iv_i)^2$, which is equal to $\|u\|^4\|v\|^4$.
 Therefore $(\gamma^2)^{\rm (sgm)}=(\gamma^2)^{\rm (mix)}$.
 Conversely, assume $(\gamma^2)^{\rm (sgm)}=(\gamma^2)^{\rm (mix)}$.
 Since $\cU$ is a non-empty finite subset,
 there exist some $u\in\cU$ and some $i\in\{1,\ldots,m\}$ such that
 \begin{align*}
  u_i>0\ \mbox{and}\ u_i\geq w_i\ (\forall w\in\cU).
 \end{align*}
 Fix such $u$ and $i$.
 We show $u\in\mbZ_i$.
 Define an integer vector $\bar{u}\in\mbZ^m$
 by $\bar{u}_i=u_i$ and $\bar{u}_j=-u_j$ for $j\neq i$.
 Since $|u_i+\bar{u}_i|=2u_i>u_i$, we have
 $\mathrm{abs}(u+\bar{u})\notin \cU\cup\{0\}$
 and therefore $\omega_{\cU}(u,\bar{u},-u,-\bar{u})=1$.
 Let $\{U_{(k)}\}_{k=1}^4$ be
 four independent Bernoulli randomization of $u$.
 Note that each $U_{(k)}$ takes $u$ (resp.\ $\bar{u}$)
 with probability at least $2^{-m}$.
 We evaluate
 \begin{align*}
  0 &= (\gamma^2)^{\rm (mix)}-(\gamma^2)^{\rm (sgm)}
  \\
  &\geq \mathrm{E}_{U_{(1)},U_{(2)},U_{(3)},U_{(4)}}
  \left[
  \omega_{\cU}(U_{(1)},U_{(2)},U_{(3)},U_{(4)})
  \left(
  1 - \frac{(U_{(1)}^{\top}U_{(2)})^2(U_{(3)}^{\top}U_{(4)})^2}{\|u\|^8}
  \right)
  \right]
  \\
  &\geq 2^{-4m}
  \omega_{\cU}(u,\bar{u},-u,-\bar{u})
  \left(
  1 - \frac{(u^{\top}\bar{u})^4}{\|u\|^8}
  \right)
  \geq 0.
 \end{align*}
 This implies $|u^{\top}\bar{u}|=\|u\|^2$.
 By equality condition of the Cauchy-Schwarz inequality,
 there is a real number $\rho$ such that $u=\rho\bar{u}$.
 This implies $u\in\mbZ_i$.
 Now, by contradiction, assume that there exists some $v\in\cU\setminus\mbZ_i$.
 We further assume $v_i\geq w_i$ for any $w\in\cU\setminus\mbZ_i$ without loss of generality.
 Since $u_i+v_i>v_i$ and $u+v\notin\mbZ_i$,
 we deduce $u+v\notin\cU\cup\{0\}$.
 Hence $\omega_{\cU}(u,v,-u,-v)=1$.
 Then we have
 \[
  0\ =\ (\gamma^2)^{\rm (mix)}-(\gamma^2)^{\rm (sgm)}
 \ \geq\ 2^{-4m}\omega_{\cU}(u,v,-u,-v)
 \left(1-\frac{(u^{\top}v)^4}{\|u\|^4\|v\|^4}\right)
 \ \geq\ 0.
 \]
 This implies $|u^{\top}v|=\|u\|\|v\|$.
 By equality condition of the Cauchy-Schwarz inequality,
 there is a real number $\tilde{\rho}$ such that $v=\tilde{\rho}u$.
 This implies $v\in\mbZ_i$ and contradict the definition of $v$.
 Thus we have $\cU\subset \mbZ_i$.

\subsection{Proof of Theorem~\ref{thm:asymptotics} and Corollary~\ref{coro:asymptotics-easy}}

We first prove Theorem~\ref{thm:asymptotics}.
Put $d=\max_{m}\max_{u\in\cU_m}|\sigma(u)|<\infty$.
We abbreviate $\cU_m$ by $\cU$ below.
It is sufficient to prove that
$(\gamma_{\cU}^2)^{\rm (sgm)}\leq |N(\cU)|$
and $(\gamma_{\cU}^2)^{\rm (mix)}\geq c|M(\cU)|$
with a positive constant $c$.
If $(u,v)\notin N(\cU)$,
then $U^{\top}V=0$ in (\ref{eqn:gamma-grad}).
Hence
\[
 (\gamma_{\cU}^2)^{\rm (sgm)}
 \leq \sum_{(u,v)\in N(\cU)}\mathrm{E}_{U,V,\tilde{U},\tilde{V}}
 \left[\omega_{\cU}(U,V,\tilde{U},\tilde{V})\frac{(U^{\top}V)^2(\tilde{U}^{\top}\tilde{V})^2}
 {\|u\|^4\|v\|^4}\right]
 \leq |N(\cU)|.
\]
We next evaluate (\ref{eqn:gamma-mix}).
If $(u,v)\in M(\cU)$, then $\omega_{\cU}(u,v,-u,-v)=1$.
Since $u$ has at most $d$ non-zero elements, 
the event $U=u$ happens with probability
at least $2^{-d}$, where $U$ is a Bernoulli randomization of $u$.
Therefore
\[
 (\gamma_{\cU}^2)^{\rm (mix)}
 \geq \sum_{(u,v)\in M(\cU)}\mathrm{E}_{U,V,\tilde{U},\tilde{V}}
 \left[\omega_{\cU}(U,V,\tilde{U},\tilde{V})\right]
 \geq 2^{-4d}|M(\cU)|.
\]
This proves Theorem~\ref{thm:asymptotics}.

Next we prove Corollary~\ref{coro:asymptotics-easy}.
Assume $|N(\cU)|/|\mu(\cU)|^2\to 0$.
Note that $|\mu(\cU)|\to\infty$ since $|N(\cU)|\geq |\cU|\geq 1$.
From the definition of $M(\cU)$ and $\mu(\cU)$,
the set $\{(u,v)\in\cU^2\mid u,v\in\mu(\cU),u\neq v\}$
is a subset of $M(\cU)$.
Then we have $|M(\cU)|\geq |\mu(\cU)|(|\mu(\cU)|-1)$.
Thus
\[
 \frac{|N(\cU)|}{|M(\cU)|}
 \leq \frac{|N(\cU)|}{|\mu(\cU)|^2(1-|\mu(\cU)|^{-1})}
 \to 0
\]
and the proof is completed.

\section*{Acknowledgements}

This study was partially supported by the Global Center of
Excellence ``The research and training center for new development in
mathematics''
and by the Ministry of Education, Science, Sports and Culture, Grant-in-Aid
for Young Scientists (B), No.~19700258.



\begin{thebibliography}{10}
\expandafter\ifx\csname natexlab\endcsname\relax\def\natexlab#1{#1}\fi
\expandafter\ifx\csname url\endcsname\relax
  \def\url#1{{\tt #1}}\fi

\bibitem[Amari(1985)]{amari1985}
S.~Amari.
\newblock {\em Differential-Geometrical Methods in Statistics}.
\newblock Springer, New York, 1985.

\bibitem[Barndorff-Nielsen(1978)]{barndorff-nielsen1978}
O.~Barndorff-Nielsen.
\newblock {\em Information and Exponential Families in Statistical Theory}.
\newblock Wiley, New York, 1978.

\bibitem[Cover and Thomas(2006)]{cover2006}
T.~M. Cover and J.~A. Thomas.
\newblock {\em Elements of Information Theory}.
\newblock John Wiley \& Sons, Inc., Hoboken, second edition, 2006.

\bibitem[Efron(1975)]{efron1975}
B.~Efron.
\newblock Defining the curvature of a statistical problem (with applications to
  second order efficiency).
\newblock {\em Ann. Statist.}, 3\penalty0 (6):\penalty0 1189--1242, 1975.

\bibitem[Marriott(2002)]{marriott2002}
P.~Marriott.
\newblock On the local geometry of mixture models.
\newblock {\em Biometrika}, 89\penalty0 (1):\penalty0 77--93, 2002.

\bibitem[McCullagh and Nelder(1989)]{GLM}
P.~McCullagh and J.~A. Nelder.
\newblock {\em Generalized Linear Models}.
\newblock Chapman and Hall/CRC, second edition, 1989.

\bibitem[Reeds(1975)]{reeds1975}
J.~Reeds.
\newblock Discussion to {E}fron's paper.
\newblock {\em Ann. Statist.}, 3\penalty0 (6):\penalty0 1234--1238, 1975.

\bibitem[Sei(2010)]{sei2010}
T.~Sei.
\newblock A structural model on a hypercube represented by optimal transport.
\newblock {\em Statistica Sinica}, 2010.
\newblock To appear. (Preprint: arXiv:0901.4715).

\bibitem[Villani(2003)]{villani2003}
C.~Villani.
\newblock {\em Topics in Optimal Transportation}.
\newblock AMS, Providence, 2003.

\bibitem[Villani(2009)]{villani2009}
C.~Villani.
\newblock {\em Optimal Transport, Old and New}.
\newblock Springer, Berlin, 2009.

\end{thebibliography}

\end{document}